\documentclass[12pt,a4paper,oneside,titlepage]{article}
\usepackage{amsmath,amsfonts,amssymb,amsthm,amscd,graphicx}
\newtheorem{Theorem}{Theorem}[section]

\newtheorem{Corollary}[Theorem]{Corollary}

\title{Metrical irrationality results related to values of the Riemann $\zeta$-function 
\footnotetext{AMS Class:11J72 .}
}
\author{Jaroslav Han\v{c}l\footnote{This work was supported by grant 
no.17-02804S of the Czech Grant Agency. } and Simon Kristensen
}

\begin{document}

\maketitle
\date{}

\begin{abstract} 
 We introduce a one-parameter family of series associated to the Riemann $\zeta$-function and prove that the values of the elements of this family at integers are linearly independent over the rationals for almost all values of the parameter, where almost all is with respect to any sufficiently nice measure. We also give similar results for the Euler--Mascheroni constant, for $\sum_{n=1}^\infty \frac{1}{n^n}$ and for   $\sum_{n=1}^\infty \frac{1}{n! +1}$. Finally, specialising the criteria used, we give some new criteria for the irrationality of $\zeta(k)$, the Euler--Mascheroni constant and the latter two series.
\end{abstract}

\section{Introduction}

The irrationality of the value of the Riemann $\zeta$-function at odd integers is a long standing open problem. For even integer arguments, it was famously shown by Euler \cite{euler} that 
\begin{equation*}
\zeta(2n) = (-1)^{n-1} \frac{2^{2n-1} \pi^{2n} B_{2n}}{(2n)!},
\end{equation*}
where $B_{2n}$ is the $2n$'th Bernoulli number, which is rational. Lindemann \cite{lindemann} proved in 1882 that $\pi$ is transcendental, so it immediately follows that $\zeta(2n)$ is irrational for each $n \in \mathbb{N}$. By contrast, the value of $\zeta$ at odd integers largely remains a mystery, and not many results were known until 1979 when Ap\'ery \cite{apery} published a proof that $\zeta(3)$ is irrational. Rivoal \cite{rivoal} subsequently proved that infinitely many odd $\zeta$-values are irrational; and Zudilin \cite{zudilin} proved that at least one of $\zeta(5), \zeta(7), \zeta(9)$ and $\zeta(11)$ is irrational.

We are not able to resolve the question of the irrationality of odd $\zeta$-values but we will present a related result. To motivate our results, let $k \ge 2$ be an integer. We first observe that 
\begin{equation*}
\zeta(k) = \sum_{n=1}^\infty \frac{1}{n^k} = \sum_{n=1}^\infty \frac{((n-1)!)^k}{(n!)^k} = \sum_{n=1}^\infty \frac{[((n-1)!)^k]}{(n!)^k},
\end{equation*}
where $[x]$ denotes the integer part of $x$, so that the last equality is trivial. We will modify this last expression by putting a real parameter inside the square brackets in the numerator, that is, we consider the series
\begin{equation}
\label{eq:modified_zeta}
\sum_{n=1}^\infty \frac{[((n-1)!)^k x]}{(n!)^k}.
\end{equation}
Below, we will prove that for any integer $k \ge 2$, this series is irrational for $\mu$-almost all $x$, whenever $\mu$ is a Radon measure with positive Fourier dimension (see the definition below). In particular, this holds for almost all $x$ with respect to Lebesgue measure. Appealing to results of Kauffman \cite{kaufman1} and \cite{kaufman2}, we immediately see that the result also holds for almost all badly approximable numbers in an appropriate sense, as well as for almost all numbers with irrationality measure greater than some prescribed $v > 2$. More details will follow below.

In addition to the perturbed $\zeta$-values of \eqref{eq:modified_zeta}, we are able to modify other famous series in a corresponding way. Vacca's formula for the Euler--Mascheroni constant $\gamma$ in \cite{vacca} states that
\begin{equation*}
\gamma = \sum_{n=1}^\infty (-1)^n \frac{[\log_2 n]}{n}.
\end{equation*}
We turn this into a factorial series as before to obtain a family of series depending on a real parameter $x$,
\begin{equation}
\label{modified_euler}
\sum_{n=1}^\infty (-1)^n \frac{[(n-1)! [\log_2 n] x]}{n!}.
\end{equation}
Again, these series are irrational almost surely with respect to any measure satisfying the above properties.

It would be natural to suspect that the full set of perturbed $\zeta$-values at integers together with the perturbed Euler--Mascheroni constant will be linearly independent over $\mathbb{Q}$ almost surely with respect to any such measure. We are not able to prove this for the particular perturbation of the series given above, although we are able to establish linear independence for the set of series
\begin{multline}
\label{eq:lin_ind}
\{\sum_{n=1}^\infty \frac{[((n-1)!)^Kn^{K-j} x]}{(n!)^K}; j\in\{ 2,\dots ,K\}\}\\ 
\cup \{\sum_{n=1}^\infty (-1)^n \frac{[((n-1)!)^Kn^{K-1} x[\log_2 n]]}{(n!)^K}, 1\},
\end{multline}
where $K \in \mathbb{N}$. These series specialise to the $\zeta$-values and the Euler--Maschero\-ni constant respectively, if we let $x=1$. Note however, that in this case the particular form of the series considered depend on their quantity.

Finally, to illustrate the versatility of the method, we make similar modifications to two other famous series, whose irrationality is at present unknown, namely the series
\begin{equation*}
\sum_{n=1}^{\infty} \frac{1}{n^n} \quad \text{ and } \quad \sum_{n=1}^\infty \frac{1}{n! + 1}.
\end{equation*}
The first series is sometimes known as Sophmore's Dream, due to the seemingly `too-good-to-be-true' identity
$$
\sum_{n=1}^{\infty} \frac{1}{n^n} = \int_0^1 x^{-x} dx,
$$
discovered by J. Bernoulli in 1697. The first terms can be found in \cite{s}. The second one is due to Erd\H{o}s \cite{Erdos2}. In fact, he asked if for any integer $t$ the sum of the series  $\sum_{n=1, n!\not= -t}^\infty \frac{1}{n! + t}$ is irrational. 
With these two series, the perturbed variants become
\begin{equation}
\label{eq:modified_series}
\sum_{n=1}^\infty \frac{\left[\prod_{j=1}^{n-1} j^j x\right]}{\prod_{j=1}^{n} j^j} \quad \text{ and } \quad \sum_{n=1}^\infty \frac{\left[\prod_{j=1}^{n-1} (j! +1) x\right]}{\prod_{j=1}^{n} (j! +1)}.
\end{equation} 

It is worth noting that our results are in a first instance metrical; the irrationality or linear independence is established for almost all real parameters in a set. However, our method also gives rise to an approach to proving the irrationality of the series in question for a particular, prescribed value of the parameter. Indeed, as the main idea of the proof is to establish the uniform distribution modulo $1$ of a certain sequence, we need only establish this in the particular case, as opposed to the `almost all' case, and in fact we can prove irrationality with a significantly weaker property than uniform distribution modulo $1$. In the final section of the paper, we give some seemingly new irrationality criteria for the original sequences.

Our method is in the spirit of Schlage-Puchta \cite{p} when he proved the irrationality of $\sum_{n=1}^\infty \frac{[n^\alpha]+1}{n!}$ for all reals $\alpha$. This result was also proved by Han\v cl and Tijdeman \cite{ht2} by a different method. For more related results, see Han\v cl and Tijdeman \cite{ht1} and \cite{ht3}-\cite{ht5}.

Throughout the  paper, we let $\mathbb N$, $\mathbb Q$ and $\mathbb R$ denote the sets of all positive integers, rational numbers and real numbers, respectively. For a real number $x$, we denote by  $[x]$, $\{ x\}$ and $\left\Vert x \right\Vert$ the integer part of $x$, the fractional part of $x$ and the distance from $x$ to the nearest integer, respectively.

\section{A result on uniform distribution modulo $1$}

 In this section, we provide the first ingredient to our irrationality results. The ideas of the proof are found in Haynes, Jensen and Kristensen \cite{hjk}, where the method is applied to a different problem and stated in a different form. We state the result in the form needed here.
  
We first need some notation. Let $\mu$ be a Radon measure on $\mathbb{R}$. The Fourier transform of $\mu$ is defined as
\begin{equation*}
\hat{\mu}(t) = \int_{-\infty}^\infty e^{-2\pi i xt}d\mu(x).
\end{equation*}
The behaviour of the Fourier transform of a measure at infinity is strongly related to the geometry of its support. Indeed, if we define the Fourier dimension of a measure $\mu$ to be
\begin{equation*}
\dim_F(\mu) = \sup\{\nu \ge 0 : \vert \hat{\mu}(t) \vert \ll (1+\vert t \vert)^{-\eta/2}\},
\end{equation*}
the Fourier dimension of $\mu$ always gives a lower bound on the Hausdorff dimension of the support of $\mu$.

For our purposes, it is only relevant that the Fourier dimension of the measures considered is positive. Examples of this of course include the Lebesgue measure on an interval, but other arithmetically interesting examples exist. Kaufman \cite{kaufman1} proved that the set
\begin{equation*}
F_M = \{x \in [0,1) \setminus \mathbb{Q} : a_n(x) \le M \text{ for all } n \in \mathbb{N}\},
\end{equation*}
supports such a measure whenever $M \ge 3$. Here, $a_n(x)$ is the $n$'th partial quotient in the simple continued fraction expansion of $x$. The result was extended to $M \ge 2$ by Queff\'elec and Ramar\'e \cite{qr}.

A further example, also due to Kaufman, is the set of numbers with a lower bound on their irrationality measure. For a real number $x$, let 
\begin{equation*}
w(x) = \sup\{w > 0 : \vert x - p/q \vert < q^{-w} \text{ for infinitely many } p/q \in \mathbb{Q}\}.
\end{equation*}
From Dirichlet's theorem in Diophantine approximation, $w(x) \ge 2$ for all irrational numbers $x$. Let $v \ge 2$. Kaufmann \cite{kaufman2} constructed a measure $\mu_v$ on the set
\begin{equation*}
W(v) = \{x \in \mathbb{R} : w(x) \ge v\},
\end{equation*}
such that $\dim_F \mu_v = \frac{2}{v}$. This coincides with the Hausdorff dimension of the set found by Jarn\'ik \cite{jarnik} and Besicovitch \cite{besicovitch}, and so is best possible.

The following result of Haynes, Jensen and Kristensen \cite{hjk} is stated in terms of the Kaufmann measure on $F_M$, but the proof only requires the Fourier dimension of the measure to be positive. We state Corollary 7 of that paper for general measures.

\begin{Theorem} \label{thm:UD}
Let $\mu$ be a Radon measure on $\mathbb{R}$ with $\dim_F \mu > 0$ and let $(a_n)$ be a sequence of real numbers such that for some $c > 0$,  $\vert a_{k} - a_j \vert \ge c$ for all $k,j \in \mathbb{N}$ with $k \neq j$. Then, $(a_n x)$ is uniformly distributed modulo $1$ for $\mu$-almost all $x \in \mathbb{R}$.
\end{Theorem}

We give a few words on the relation between the above statement and that of \cite{hjk}. In \cite{hjk}, the sequence $(a_n)$ is assumed to be composed of integers. This will not be the case for our sequences below, but in order for the proof of \cite{hjk} to work, we only need for the sequence to take its values in a discrete subset of the real numbers, i.e. a set with only isolated points. This is guaranteed by the assumption on universally lower bounded gaps. 

Also, in \cite{hjk}, a bound on the discrepancy of the sequence $(a_n x)$ is obtained. This gives a quantitative variant of uniform distribution modulo $1$, which we will not be needing here. Note however that the faster the sequence $(a_n)$ increases, the better the discrepancy bound.

It is curious to remark how Theorem \ref{thm:UD} yields a short proof of a result usually attributed to Kahane and Salem \cite{ks}, stating that the ternary Cantor set does not support a Radon measure with positive Fourier dimension. Indeed, suppose such a measure $\mu$ existed. By Theorem \ref{thm:UD} applied with $a_n = 3^n$ would imply that for almost all numbers $x$ in the ternary Cantor set, $(3^n x)$ would be uniformly distributed modulo $1$, which is the same as saying that almost all numbers in the ternary Cantor set are normal to base $3$. Clearly this is not the case, which completes the proof of the result of Kahane and Salem.

\section{Metrical Results}

We proceed with the announced application of uniform distribution to irrationality. The idea of using uniform distribution in proofs of irrationality appears in \cite{p}. Our approach is inspired by this paper. We begin with the announced result on linear independence of the set given in \eqref{eq:lin_ind}.

\begin{Theorem} \label{HanKris1.t1}
Let $K$ be a positive integer and let $\mu$ be a Radon measure on $\mathbb{R}$ with positive Fourier dimension.
For $\mu$-almost all numbers $x$ the set $\{\sum_{n=1}^\infty \frac{[((n-1)!)^Kn^{K-j} x]}{(n!)^K}; j\in\{ 2,\dots ,K\}\}\cup \{\sum_{n=1}^\infty (-1)^n \frac{[((n-1)!)^Kn^{K-1} x[\log_2 n]]}{(n!)^K}, 1\}$ consists of linearly independent numbers over rational numbers.
\end{Theorem}

\begin{proof}
Let $x$ be a real number. Then, there are $A_0,\dots A_K\in \mathbb Z$ not all equal to $0$ and such that 
\begin{multline} \label{HanKris1.1}
\sum_{j=2}^KA_j\sum_{n=1}^\infty\frac{[((n-1)!)^Kn^{K-j} x]}{(n!)^K}\\ +A_1\sum_{n=1}^\infty  (-1)^n \frac{[((n-1)!)^Kn^{K-1} x[\log_2 n]]}{(n!)^K}+A_0=0.
\end{multline}

Let $N\in\mathbb Z^+$. Multiplying (\ref{HanKris1.1}) by $(N!)^K$, we obtain that
\begin{multline*}
\sum_{j=2}^KA_j\sum_{n=N+1}^\infty\frac{[((n-1)!)^Kn^{K-j} x]}{((N+1)\dots n)^K}\\
+A_1\sum_{n=N+1}^\infty  (-1)^n \frac{[((n-1)!)^Kn^{K-1} x[\log_2 n]]}{((N+1)\dots n)^K}+B=0,
\end{multline*}
where $B$ is a suitable integer constant which depends on $N$.

The sequences in this expression converge, and both
$$
\sum_{j=2}^KA_j\sum_{n=N^{N(K+1)}+1}^\infty\frac{[((n-1)!)^Kn^{K-j} x]}{((N+1)\dots n)^K} = O\left(\frac{1}{N}\right)
$$ 
and 
$$
A_1\sum_{n=N^{N(K+1)}+1}^\infty  (-1)^n \frac{[((n-1)!)^Kn^{K-1} x[\log_2 n]]}{((N+1)\dots n)^K}= O\left(\frac{1}{N}\right)
$$
We remove these tails at the cost of introducing an error term of order $O(\frac{1}{N})$.

Now, note that $[x] = x - \{x\} = x + O(1)$ and apply this to remove the integer part of the sequence of numerators at the cost of a very small error, which is absorbed in the $O(\frac{1}{N})$. The upshot is that
\begin{equation}
\begin{split}
\sum_{j=2}^KA_j\sum_{n=N+1}^{N^{N(K+1)}}\frac{((n-1)!)^Kn^{K-j} x}{((N+1)\dots n)^K}&+A_1\sum_{n=N+1}^{N^{N(K+1)}}  (-1)^n \frac{((n-1)!)^Kn^{K-1} x[\log_2 n]}{((N+1)\dots n)^K}\\
&+
\label{eq:limit_zero}
B+O\left(\frac 1N\right)=0. 
\end{split}
\end{equation}
As $N$ tends to infinity, the error term vanishes, and as $B$ is an integer, the fractional part of the first expression must converge to $0$.
 
However,  we will see that by Theorem \ref{thm:UD}, the sequence
\begin{alignat*}{2}
(f_x(N)) =&
\Bigg(\sum_{j=2}^KA_j\sum_{n=N+1}^{N^{N(K+1)}}\frac{((n-1)!)^Kn^{K-j} x}{((N+1)\dots n)^K}\\
&+A_1\sum_{n=N+1}^{N^{N(K+1)}}  (-1)^n \frac{((n-1)!)^Kn^{K-1} x[\log_2 n]}{((N+1)\dots n)^K}\Bigg)
\end{alignat*}
is uniformly distributed modulo $1$ for $\mu$-almost all $x$. This will complete the proof, as it immediately implies that the expression in \eqref{eq:limit_zero} cannot tend to an integer.

To apply Theorem \ref{thm:UD}, we need to check that the sequence of integers considered satisfies the appropriate conditions. Namely, we need to check that the sequence $(a_N)$ given by 
\begin{multline*}
a_N =\sum_{j=2}^KA_j\sum_{n=N+1}^{N^{N(K+1)}}\frac{((n-1)!)^Kn^{K-j} }{((N+1)\dots n)^K}\\
+A_1\sum_{n=N+1}^{N^{N(K+1)}}  (-1)^n \frac{((n-1)!)^Kn^{K-1} [\log_2 n]}{((N+1)\dots n)^K}
\end{multline*}
is a discrete subset of the reals. This is however simple. Each of the interior sums in the first term is equal to 
\begin{equation}
\label{eq:zetaUD}
(N!)^K \sum_{n=N+1}^{N^{N(K+1)}}\frac{1}{n^j},
\end{equation}
which grows at least as fast as $N!$ to any power slightly smaller than $k_j$, and so the asymptotic growth of the sequence $a_N$ is governed by the largest value of $k_j$, for which it grows in absolute value like a power of $N!$. Clearly, this sequence has the desired property. 

If $K=1$, and only the perturbed Euler--Mascheroni constant is considered, we obtain similarly a very rapid growth in the $a_N$.
\end{proof}

For the first perturbations of the series expansion of $\zeta$, we have the following almost sure irrationality statement.

\begin{Theorem}
Let $\mu$ be a Radon measure on $\mathbb{R}$ with positive Fourier exponent and let $\alpha \ge 2$ be an integer.
For $\mu$-almost all numbers $x$ the sum $\sum_{n=1}^\infty \frac{[((n-1)!)^\alpha x]}{(n!)^\alpha}$ as well as the sum $\sum_{n=1}^\infty (-1)^n \frac{[(n-1)! x[\log_2 n]]}{n!}$ are irrational numbers.
\end{Theorem}

\begin{proof}
Fix one of the series, $\sum_{n=1}^\infty \frac{[((n-1)!)^\alpha x]}{(n!)^\alpha}$ say. Suppose to the contrary that the series is rational and pick $p,q \in \mathbb{N}$ such that
\begin{equation}
\label{eq:zeta_1}
q \sum_{n=1}^\infty \frac{[((n-1)!)^\alpha x]}{(n!)^\alpha} = p.
\end{equation}
Let $N \in \mathbb{N}$ and multiply \eqref{eq:zeta_1} by $(N!)^\alpha$. We then find that
$$
q \sum_{n=N!+1}^\infty \frac{[((n-1)!)^\alpha x]}{(n!)^\alpha} = B
$$
for some $B \in \mathbb{Z}$.

We truncate the series at $(N!)^3$. Estimating the remainder by an integral, we easily find that
\begin{multline*}
q  \sum_{n=(N!)^3+1}^\infty \frac{[((n-1)!)^\alpha x]}{(n!)^\alpha} \le qx (N!)^\alpha \sum_{n=(N!)^3+1}^\infty \frac{1}{N^\alpha} \\
 \ll qx (N!)^{\alpha - 3 \alpha + 3} = O\left(\frac1N\right).
\end{multline*}

We now apply the property that $[x] = x - \{x\} = x + O(1)$ to remove the integer part in what remains, noting that
$$
\sum_{n=N!+1}^{(N!)^3} \frac{1}{(n!)^\alpha} = O\left(\frac1N\right).
$$

The upshot is that
\begin{equation}
\label{eq:zeta_2}
\left\Vert x q \sum_{n=N!+1}^{(N!)^3} \frac{((n-1)!)^\alpha}{(n!)^\alpha} \right\Vert = O\left(\frac1N\right).
\end{equation}
But clearly the sequence 
$$
\left(q \sum_{n=N!+1}^{(N!)^3} \frac{((n-1)!)^\alpha}{(n!)^\alpha}\right)
$$
satisfies the assumptions of Theorem \ref{thm:UD}, so that the interior of \eqref{eq:zeta_2} is uniformly distributed modulo $1$ for $\mu$-almost all $x$. In particular, it is close to $\frac{1}{2}$ infinitely often, which is a contradiction.

For the perturbed Euler--Mascheroni constant, the same method and truncation applies.
\end{proof}

Note that we cannot prove the almost sure linear independence of these series by the present method. The reason is simple: As we remove the square brackets to pass from integer part in the numerator to the different series to which Theorem \ref{thm:UD} is applicable, since the exponents in the denominators in the last result are all different, we would get an error which is too large to be useful at all. 

We now prove almost sure irrationality of the last two perturbed series.

\begin{Theorem} \label{HanKris1.t2}
Let $\mu$ be a Radon measure on $\mathbb{R}$ with positive Fourier dimension.
For $\mu$-almost all numbers $x$ the number $\sum_{n=1}^\infty \frac{[\prod_{j=1}^{n-1}j^j x]}{\prod_{j=1}^nj^j }$ is irrational.
\end{Theorem}

\begin{proof}
Suppose the contrary that the series is rational. Let $x$ be a real number. Let $p,q\in\mathbb Z^+$ such that 

\begin{equation} \label{HanKris1.2}
q\sum_{n=1}^\infty \frac{[\prod_{j=1}^{n-1}j^j x]}{\prod_{j=1}^nj^j }=p.
\end{equation}

Let $N\in\mathbb Z^+$ and multiply \eqref{HanKris1.2} by $\prod_{j=1}^Nj^j$ to obtain 
$$
q\sum_{n=N+1}^\infty \frac{[\prod_{j=1}^{n-1}j^j x]}{\prod_{j=N+1}^nj^j }=B.
$$
where $B$ is a suitable integer constant which depends on $N$. We now truncate at $N^2+1$ and observe that
$$
q\sum_{n=N+1}^\infty \frac{\{\prod_{j=1}^{n-1}j^j x\}}{\prod_{j=N+1}^nj^j }  = O\left( \frac 1N\right), \quad  q\sum_{n=N^2+1}^\infty \frac{[\prod_{j=1}^{n-1}j^j x]}{\prod_{j=N+1}^nj^j } = O\left( \frac 1N\right).
$$
As in the preceding proof, we remove the integer part from the remaining term and find that,
\begin{equation}
\label{eq:UD_contradiction}
\left\{  q\sum_{n=N+1}^{n=N^2} \frac{\prod_{j=1}^{n-1}j^j x}{\prod_{j=N+1}^nj^j } \right\} =O(\frac 1N). 
\end{equation}

Now, the sequence
$$
 (a_N) =  \left(q\sum_{n=N+1}^{n=N^2} \frac{\prod_{j=1}^{n-1}j^j }{\prod_{j=N+1}^nj^j }\right)
 $$ 
is an increasing sequence of rationals taking values in a discrete set, so by Theorem \ref{thm:UD}, the sequence $\{a_Nx\}$ is uniformly distributed modulo $1$. This is in contradiction with \eqref{eq:UD_contradiction}.
\end{proof}

We finish this section with the perturbed sum for $\sum \frac{1}{n!+1}$.

\begin{Theorem} \label{HanKris1.t3}
Let $\mu$ be a Radon measure on $\mathbb{R}$ with positive Fourier dimension.
Then for almost all numbers $x$ the number $\sum_{n=1}^\infty \frac{[\prod_{j=1}^{n-1}(j!+1) x]}{\prod_{j=1}^n(j!+1) }$ is irrational.
\end{Theorem}

\begin{proof}
Suppose the contrary. Let $x$ be a real number. Let $p,q\in\mathbb Z^+$ such that 
\begin{equation} \label{HanKris1.3}
q\sum_{n=1}^\infty \frac{[\prod_{j=1}^{n-1}(j!+1) x]}{\prod_{j=1}^n(j!+1) }=p.
\end{equation}
  
Let $N\in\mathbb Z^+$ and multilpy \eqref{HanKris1.3} by $\prod_{j=1}^N(j!+1)$ to obtain
$$
q\sum_{n=N+1}^\infty \frac{[\prod_{j=1}^{n-1}(j!+1) x]}{\prod_{j=N+1}^n(j!+1) }=B.
$$
where $B$ is a suitable integer constant which depends on $N$. We truncate again at $n=N^2+1$ and note that 
$$
q\sum_{n=N+1}^\infty \frac{\{\prod_{j=1}^{n-1}(j!+1) x\}}{\prod_{j=N+1}^n(j!+1) }= O\left(\frac 1N \right)$$
and
$$ 
\quad q\sum_{n=N^2+1}^\infty \frac{[\prod_{j=1}^{n-1}(j!+1) x]}{\prod_{j=N+1}^n(j!+1) }= O\left(\frac 1N \right).
$$
As before, this implies that
\begin{equation}
\label{eq:UD_contradiction2}
\left\{ q\sum_{n=N+1}^\infty \frac{\{\prod_{j=1}^{n-1}(j!+1) x\}}{\prod_{j=N+1}^n(j!+1) }\right\} = O(\frac 1N). 
\end{equation}

To obtain a contradiction, we need only note that 
$$
(a_N) = \left(q\sum_{n=N+1}^{n=N^2+1} \frac{\prod_{j=1}^{n-1}(j!+1) }{\prod_{j=N+1}^n(j!+1) }\right)
$$
is an increasing sequence of rationals taking values in a discrete set and apply Theorem \ref{thm:UD}.
\end{proof}

\section{Criteria for irrationality}

In the above proofs, we have used much stronger results than actually needed. In fact, the uniform distribution of the sequences in question is unnecessarily strong, and we only need for the sequences of fractional parts in the proofs to have an accumulation point which is not an integer. 

Inserting $x=1$ in the various proofs recovers the original series, and in this way, we obtain some seemingly new criteria for the irrationality of the various series. This is where the explicit value of the truncation point is needed. We state these as corollaries.

\begin{Corollary}
Let $k \ge 2$ be an integer. If the sequence
$$
\left(\left\{\sum_{n=N+1}^{(N!)^{(2k-1)/(k-1)}} \frac{((n-1)!)^{k}}{((N+1)\dots n)^{k}}\right\}\right)
$$
has an irrational accumulation point or infinitely many accumulation points then $\zeta(k)$ is irrational.
\end{Corollary}

It is tempting to conduct numerical experiments on the distribution of this sequence for some value of $k$. With the help of Alex Ghitza, we have run some numerical experiments on $\zeta(5)$ using \texttt{Sage}. It does not appear that the sequence arising from $k=5$ accumulates at the endpoints of the unit interval. This is however not surprising, as is seen from \eqref{eq:zetaUD}. Indeed, from a numerical point of view, the sum $\sum_{n=N+1}^{(N!)^{9/4}}\frac{1}{n^5}$ is virtually indistinguishable from the sum $\sum_{n=N+1}^{\infty}\frac{1}{n^5}$. On multiplying by $(N!)^5$ and adding the integer $(N!)^5\sum_{n=1}^{(N)!}\frac{1}{n^5}$, which makes no difference as we are considering the sequence modulo $1$, numerically we are in fact just seeing the fractional parts of the sequence $(N!)^5 \zeta(5)$, for which the criterion is clear: if $\zeta(5)$ is rational, this sequence would be an integer for $N$ larger than the denominator of $\zeta(5)$.

We state the corresponding results for the Euler--Mascheroni constant and the remaining two series.

\begin{Corollary}
If the sequence
$$
\left(\left\{\sum_{n=N+1}^{(N!)^3} (-1)^n \frac{(n-1)! [\log_2 n]}{(N+1)\dots n}\right\}\right)
$$
has an irrational accumulation point of infinitely many accumulation points then the Euler--Mascheroni constant $\gamma$ is irrational.
\end{Corollary}

From the numerical point of view, this sequence has the same defect as the preceding ones, and one would just end up with an experiment on the original Euler--Mascheroni constant.

The irrationality criteria for the Sophmore's Dream problem and the Erd\H{o}s problem are given in the following two corollaries.

\begin{Corollary}
If the sequence
$$
\left(\left\{  \sum_{n=N+1}^{N^2} \frac{\prod_{j=1}^{n-1}j^j }{\prod_{j=N+1}^nj^j } \right\}\right) = 
\left(\left\{  \sum_{n=N+1}^{N^2} \frac{\prod_{j=1}^Nj^j }{n^n } \right\}\right)
$$
has an irrational accumulation point or infinitely many accumulation points, then the series $\sum_{N=1}^\infty \frac1{n^n}$ is irrational.
\end{Corollary}

\begin{Corollary}
If the sequence
$$
\left(\left\{ \sum_{n=N+1}^{n=N!+1} \frac{\prod_{j=1}^{n-1}(j!+1) }{\prod_{j=N+1}^n(j!+1) }\right\}\right) =
\left(\left\{ \sum_{n=N+1}^{n=N!+1} \frac{\prod_{j=1}^N(j!+1) }{n!+1 }\right\}\right)
$$
has infinitely many accumulation points or an irrational accumulation point, then the series $\sum_{N=1}^\infty \frac1{n!+1}$ is irrational.
\end{Corollary}

Numerically these are less unweildy than the series related to the $\zeta$-function. Nonetheless, the numbers involved grow extremely rapidly, and we have not been able to get any useful information from numerical experimentation.

As a final remark, one can also obtain a criterion for the linear independence of $\zeta$-values from the above, although this is slightly more convoluted. Concretely, we get the following.

\begin{Corollary}
Let $K$ be a positive integer. Suppose that for any choice of $A_1, \dots, A_K \in \mathbb{Z}$ not all equal to $0$, the sequence of fractional parts of
$$
\sum_{j=2}^KA_j\sum_{n=N+1}^{N^{N(K+1)}}\frac{((n-1)!)^Kn^{K-j} }{((N+1)\dots n)^K}\\
+A_1\sum_{n=N+1}^{N^{N(K+1)}}  (-1)^n \frac{((n-1)!)^Kn^{K-1} [\log_2 n]}{((N+1)\dots n)^K}
$$
has an accumulation point different from $0$ and $1$.
Then, the set $$\{\gamma, \zeta(2), \zeta(3), \dots, \zeta(K)\}$$ consists of linearly independent numbers over rational numbers.
\end{Corollary}

\paragraph{Achnowledgements:} We thank Alex Ghitza for helping us with \texttt{Sage}.

Department of Mathematics, Faculty of Science, University of Ostrava, 30.~dubna~22, 701~03 Ostrava~1, Czech Republic.\\
e-mail: hancl@osu.cz,

Department of Mathematics, Aarhus University, Ny Munkegade 118, 8000 Aarhus, Denmark.\\
e-mail: sik@math.au.dk,  

\end{document}